\documentclass[12pt]{article}
\usepackage{amsfonts,amssymb,latexsym, amsmath, amsthm}

\usepackage[usenames]{color}%%  NEW
\newtheorem{thm}{Theorem}[section]
\newtheorem{lem}[thm]{Lemma}

\newtheorem{proposition}[thm]{Proposition}
\newtheorem{corollary}[thm]{Corollary}

\newcounter{unnumber}

\newtheorem{nothm}[unnumber]{Theorem$\!\!$}

\newtheorem{qu}{Question}

\theoremstyle{remark}

\theoremstyle{plain}

\numberwithin{equation}{section}

\def\ZZ{{\mathbb Z}}

\def\QQ{{\mathbb Q}}
\def\RR{{\mathbb R}}
\def\NN{{\mathbb N}}

\def\ab{{\rm ab}}

\def\leq{\leqslant}
\def\geq{\geqslant}

\def\Ntf{{{\mathcal N}_{\mathrm {tf}}}}

\newcommand{\gp}{{\langle}}
\newcommand{\rp}{{\rangle}}
\newcommand{\ga}{\gamma}

\def\ZZ{{\mathbb Z}}

\def\QQ{{\mathbb Q}}
\def\RR{{\mathbb R}}
\def\NN{{\mathbb N}}

\def\leq{\leqslant}
\def\geq{\geqslant}

\def\gt{$\hat {\rm t}$}
\def\C{{\mathrm C}}

\newtheorem{defi}{Definition}[section]%master template; adds %section
\newtheorem{theorem}{\bf Theorem}

\newtheorem{prop}[defi]{Proposition}

\newtheorem{cor}[defi]{Corollary}

\newtheorem{remark}[defi]{Remark}
\newtheorem{example}[defi]{Example}
\newtheorem{*example}[defi]{*Example}

\addtocontents{ences.txt}{References}

\begin{document}
\title{\bf Residual nilpotence and ordering in one-relator groups and knot groups}
\author{I.\ M.\ Chiswell, A.\ M.\ W.\ Glass and John S.\ Wilson} \small{\date{\today}} \maketitle  \def\thefootnote{} \footnote{

%2010 AMS Classification: 20F10, 06F15, 20F60.

 \vfill\eject
The authors are most grateful to Queens' College, Cambridge for partial funding and Dale Rolfsen for very helpful comments.  The third author thanks the Erwin Schr\"odinger Institut in Vienna,
where part of this work was done.
}

\setcounter{theorem}{0}
\setcounter{footnote}{0}

\begin{abstract}
Let $G=\langle x,t\mid w\rangle$ be a one-relator group, where $w$ is a word in $x,t$. If $w$ is a product of conjugates of $x$ then, associated with $w$, there is a polynomial $A_w(X)$ over the integers, which in the case when $G$ is a knot group, is the Alexander polynomial of the knot.  We prove, subject to certain restrictions on $w$, that if all roots of $A_w(X)$ are real and  positive then $G$ is bi-orderable, and that if $G$ is bi-orderable then at least one root is real and positive. This sheds light on the bi-orderability of certain knot groups and on a question of Clay and Rolfsen. One of the results relies on an extension of work of G. Baumslag on adjunction of roots to groups,
and this may have independent interest.
\end{abstract}

\section{Introduction}

Let $G$ be a group.   A total order $<$ on $G$ is called a {\em bi-order}
if $u<v$ implies $ux<vx$ and $xu<xv$ for all $u,v,x\in G$.  We say that $G$ is {\em  bi-orderable} if there exists a bi-order on $G$.  Free groups are bi-orderable (see \cite[Example 1.3.24]{G99}).
The bi-orderability of groups arising in a topological context, especially knot groups, has attracted considerable interest.  Here we shall be mainly concerned with one-relator groups with two generators.

Let $w=w(x,t)$ be a non-identity element of the free group $F$ on $x,t$, and suppose that the weight (that is, the exponent sum) of $t$ in $w$ is zero. Then $w$ lies in the normal subgroup $N$ of $F$ generated by $x$ and we can write
$$w=x^{m_1t^{d_1}}\dots
x^{m_rt^{d_r}}\eqno(\ast)$$ for some integer $r\geq1$ and some integers $d_i$, $m_i$.
(Throughout, $a^b$ stands for $b^{-1}ab$, $[a,b]$ for $a^{-1}a^b$.) For each $j\in\ZZ$, let $\tau_j(w)=\{i\mid d_i=j\}$ and let $S_w=\{j\mid\sum_{i\in \tau_j(w)}m_i\neq0\}$. Thus $S_w$ is empty if and only if $w$ is in the derived subgroup $N'$ of $N$.
Suppose that $w\in N\setminus N'$, so that $S_w\neq\emptyset$; let $e_w$ and $d_w$ be the smallest and largest members of $S_w$ and define
$A_w(X)$ to be the element $\sum_{i=1}^r m_iX^{d_i-e_w}$ of the polynomial ring $\ZZ[X]$. Hence $A_w(X)$ is a polynomial with non-zero constant coefficient.

We say that $w$ is a {\em tidy word} if $w\in N\setminus N'$ and $\tau_j(w)=\emptyset$ for all $j>d_w$ and all $j<e_w$. We say that $w$ is {\em principal} if $w$ is  a tidy word and $\tau_{d_w}(w)$ contains just one element; $w$ is called {\em monic} if $w$ is principal and $m_k=1$, where $k$ is the unique element of $\tau_{d_w}(w)$. Thus the polynomial $A_w(X)$ is monic if $w$ is monic. Clearly $A_w=A_{w^t}$.

We shall consider groups $G$ with two generators $x,t$ and a single defining relation $w(x,t)=1$.
By possibly a change of generators (cf.\ \cite[Chapter V, Lemma 11.8]{LS}), we may assume that $w$ has the form $(\ast)$. Let $K$ be the normal subgroup of $G$ generated by $x$; so $G=K\rtimes \langle t\rangle$ and $\langle t\rangle$ is infinite cyclic. Note that, assuming $w\in N\setminus N'$,
$A_w(X)$ is an integer multiple of the characteristic polynomial of the automorphism of
$(K/K')\otimes_{\ZZ}\RR$ induced by conjugation by $t$ in $G$. Moreover, the characteristic polynomial
of this automorphism is equal to its minimal polynomial.

In the case when $G=\langle x,t\mid w\rangle$ is the group of a knot in the $3$-sphere, $A_w$ is just the Alexander polynomial of the knot. This can be seen, for example, from the description of the Alexander polynomial in \cite[Chapter VIII, Section 3]{CF}.

We shall prove the following result.

\begin{theorem}\label{A} Let $h,z$ be non-identity elements of any bi-ordered group $G$ and
suppose that $w(h,z)=1$, where
$w$ is a tidy word in the free group on $x,t$.
Then $A_w(X)$ has a positive root.  \end{theorem}

In this result, $h,z$ are arbitrary elements of any bi-ordered group satisfying $w(h,z)=1$.
Theorem A may be compared with \cite[Theorem 1.4]{CR}.  It applies to knot groups and generalizes \cite[Theorem 1.1]{CR} in the $2$-generator case.

The main result in \cite{PR} asserts that, if the Alexander polynomial of a fibred knot has all its roots real and positive, then its group is bi-orderable.
It is obtained as a consequence of the following

\begin{nothm}[\cite{PR}, Theorem 2.6]
Let $\varphi$ be an automorphism of a free group $F$ of finite rank.
If all eigenvalues of the automorphism of the abelianization $F^{\rm{ab}}$ of $F$ induced by $\varphi$ are real and positive, then $F$ has a $\varphi$-invariant bi-ordering.
\end{nothm}

This statement is equivalent to the assertion that the semi-direct product $F\rtimes\langle \varphi \rangle$ (where $\varphi$ has infinite order or is trivial) is bi-orderable (see for example Remark \ref{rem}).  We shall use the following more general criterion for bi-orderability in our key results.

\begin{theorem}\label{resnilptoorder}  Let $G$ be an extension of a residually $($torsion-free nilpotent\/$)$ group K by a bi-orderable abelian group $\Phi$.
Suppose that the real vector space $V=K^{\rm ab}\otimes \RR$ is finite-dimensional, and that all eigenvalues of maps induced on $V$ by elements of $\Phi$ are positive real numbers.  Then $G$ is bi-orderable.
\end{theorem}

The case of Theorem B in which $\Phi$ is cyclic is essentially \cite[Proposition 3.4]{LRR}. The next theorem generalizes \cite[Proposition 4.4]{PR} in the $2$-generator case.

\begin{theorem}\label{B}
Let $G$ be a group with presentation $\langle x,t\mid w\rangle$, where $w$ is a monic word in the free group on $x,t$.
If all roots of $A_w(X)$ are real and positive, then $G$ is bi-orderable.
\end{theorem}

There is a similar result for the case when $A_w(X)$ is principal but not monic.  It requires extra hypotheses that are easy to check.

\begin{theorem}\label{nm}
Let $G$ be a group with presentation $\langle x,t\mid w\rangle$, where $w$ is a principal word in the free group on $x,t$.
Let $A_w(X)=a_0+\dots+ a_{d-1}X^{d-1} -mX^d$ and assume that $\gcd\{ a_0,\dots,a_{d-1}\}=1$ and $a_{d-1}$ is not divisible by $m$.
If all roots of $A_w(X)$ are real and positive, then $G$ is bi-orderable.
\end{theorem}

Because of Theorem \ref{resnilptoorder}, in order to prove Theorems \ref{B} and \ref{nm}, it will suffice to prove that in each case the normal subgroup $K$ generated by $x\in G$ is residually (torsion-free nilpotent).

We prove Theorems A, B and C in Sections 3, 4  and 5 respectively.  Theorem D is substantially harder to prove than Theorem C, and some prerequisite results on adjunction of solutions of equations over groups are given in Section 6;  this section can be read independently. A result that implies Theorem D is proved in the final section.

 \section{Examples and applications}

%Whole section refashioned.

The conditions on words arising in Theorems C and D appear at first sight quite restrictive.  We begin with examples illustrating the need for some of the conditions, and then discuss applications of our results to knot groups. \medskip

We recall that a non-identity element $g$ of a group $G$ is a {\em generalized torsion element} (or a {\em \gt-element} for short) if some (non-empty) product of conjugates of $g$ is trivial.  Clearly bi-orderable groups have no \gt-elements; see \cite[p.\ 79]{MR}.

\begin{example} \label{e1}
{\rm Since $[a^m,b^n]$ is a product of conjugates of $[a,b]$ if $m,n>0$, the groups $G(m,n)=
\gp x,t\mid [x^m,t^n]\rp$  with $mn>1$ have $[x,t]$ as a \gt-element and so are not bi-orderable.
In this case, $A_w(X)=m(X^n-1)$ where $w=[x^m,t^n]$; so $A_w(X)$ is non-monic unless $m=1$ and has roots that are not positive real numbers unless $n=1$.}
\end{example}

\begin{example}\label{e4} {\rm If $a$, $b$ are elements of a bi-orderable group having non-trivial powers that commute, then $a$, $b$ commute.  Thus if $|m|>1$, $|n|>1$ and $w_1$, $w_2$ are tidy words that do not commute in the group $G=\langle x,t\mid w_1^{n}w_2^{-m} \rangle$, then $G$ is not bi-orderable. In particular, the Baumslag--Solitar groups $B(m,n):=\langle x,t\mid x^n(x^t)^{-m}\rangle$ with coprime $m,n\geq2$ are not bi-orderable.
Similarly, if $f(X)=b_0+\dots +b_{d-1}X^{d-1}$ is a non-zero primitive polynomial and $w=(x^{f(t)})^n(x^{t^d})^{-m}$ with coprime $m,n\geq 2$, then the group $G:=\langle x,t\mid w\rangle$ is not bi-orderable. This follows because $[x^{t^d},x^{f(t)}]\neq 1$, since
the subgroup $\langle x,x^t,\dots,x^{t^d}\rangle$ is isomorphic to $F\ast_{(x^{f(t)})^n=(x^{t^d})^m} \langle x^{t^d}\rangle$ where $F$ is the free group on generators $x,\dots,x^{t^{d-1}}$ (see the proof of the Freiheitssatz in \cite[Chapter IV, pp.\ 198--199]{LS}).  Note that $A_w(X)=nb_0+ \dots + nb_{d-1}X^{d-1} -mX^d$. Therefore the greatest common divisor condition in Theorem \ref{nm} cannot in general be removed. }  \end{example}

The next two examples indicate that the existence of one positive root does not suffice for bi-orderability.

\begin{example}
{\rm Let $G=\langle x,t\mid w\rangle$, where $w=(x^tx^m)^t(x^tx^m)^n$ for some non-zero $m,n\in \ZZ$; thus $A_w(X)=(X+m)(X+n)$.
We can write $G$ as an HNN extension with $\langle x,x^t\rangle$ a free rank $2$ subgroup of the base group (again by the proof of the Freiheitssatz, {\it op cit.}). So $x^tx^m, [x^t,x], x^tx^{-1}$ are all non-identity elements.
Hence $x^tx^m$ is a \gt-element if $n>0$.
If $n<0<m$, then $[x^t,x]^t=[(x^tx^m)^{-n}x^{-mt},x^t]$, a product of conjugates of $[x,x^t]=[x^t,x]^{-1}$ since $-n>0$;  so $[x^t,x]$ is a \gt-element.  Thus for these cases when the quadratic $A_w(X)$ has only one positive real root, the group $G$ has \gt-elements, and $G$ is not bi-orderable.
In particular, $\langle x,t\mid x^{t^2}x^{-1}\rangle$ is not bi-orderable since $x^{t^2}x^{-1}=(x^tx^{-1})^t(x^tx^{-1})$ and $x^tx^{-1}$ is a \gt-element.}
\end{example}

\begin{example}\label{e41}
{\rm Let $G:=\langle x,t\mid w\rangle$, where $w=x^{t^3}x^{-1}$.  Then $A_w(X)=X^3-1$. We have
$w=(x^tx^{-1})^{t^2}(x^tx^{-1})^t(x^tx^{-1})$; therefore $[t,x^{-1}]=x^tx^{-1}$ is a \gt-element and $G$ is not bi-orderable.}
\end{example}

\begin{qu} \rm Suppose
that $\langle x,t\mid w\rangle$ is bi-orderable. Must all roots of $A_w(X)$ be positive real numbers?
\end{qu}

\begin{qu} \rm
Let $w(x,t)$ be a tidy word that is not principal.
If all roots of $A_w(X)$ are positive real numbers, is the group $\langle x,t\mid w\rangle$ bi-orderable?
\end{qu}

\begin{qu} \rm
Is every generalized torsion-free one-relator group bi-orderable?
\end{qu}

For background material on knot groups, see \cite{Ro}.  If a knot group is a $2$-generator group and has a presentation
$\langle x,t\mid w\rangle$ with $w$ a tidy word in $x$, $t$, then, as noted in \S1, the Alexander polynomial of the knot
is equal to $A_w(X)$.
%If a knot group is a $2$-generator group, then it has a presentation $\langle x,t\mid w\rangle$ with $w$ a tidy (but not necessarily principal) word in $x$, $t$, and the Alexander polynomial of the knot is equal to $A_w(X)$.
As an immediate consequence of Theorems \ref{A}, \ref{B} and \ref{nm}, we have

\begin{cor}\label{K}
Let $G$ be the group of a knot in $S^3$.  Suppose that $G$ has a presentation with two generators and one relation which is tidy.
\begin{enumerate}
\item[\rm(1)] If $G$ is bi-orderable, then its Alexander polynomial has a positive real root.
\item[\rm(2)] If a word defining $G$ is monic and all roots of the Alexander polynomial are
positive, then $G$ is bi-orderable.
\item[\rm(3)] Suppose that the word $w$ defining $G$ is principal, and that the Alexander polynomial has the form
$a_0+\dots+ a_{d-1}X^{d-1} -mX^d$ with  $\gcd\{ a_0,\dots,a_{d-1}\}=1$ and $a_{d-1}$ not divisible by $m$. If  all roots of the Alexander polynomial are
positive, then $G$ is bi-orderable.
\end{enumerate}\end{cor}

Assertion (1) had been conjectured, and (2) was proved in \cite{PR}  for the special case of fibred knots; in this case the Alexander polynomial is monic. Assertion (3), and a slightly more technical result proved as Theorem \ref{nonmon} below,
allow us to consider some knot groups for which the Alexander polynomial is not monic.

The relator for the $5_2$ knot is $x^{-3t}x^2x^{2t^2}$.
So its Alexander polynomial is $2X^2-3X+2$.
Since this has no real roots, the $5_2$ knot group is not bi-orderable.
Generalized torsion in knot groups is studied in Naylor and Rolfsen \cite{N}, and they
establish a stronger result by exhibiting a $\hat t$-element in the $5_2$ knot group.

Quartic monic Alexander polynomials of fibred knot groups have the form $X^4-aX^3+(2a-1)X^2-aX+1$ or $X^4-aX^3+(2a-3)X^2-aX+1$ where $a\in \ZZ$
(see \cite{PR}).   For
$a=5,4,3,2,0,-1,-2$ the polynomial of the first type has no positive real roots and so by Corollary \ref{K}(1) the corresponding group is not bi-orderable.  In the other cases
our results add nothing to what was previously known (see \cite{CR}).

Further applications of Corollary \ref{K} to knot groups can be found in \cite{CDN}.

\section{Proof of Theorem \ref{A}}

Recall that the convex subgroups of a bi-ordered group form a complete
chain from $\{1\}$ to $G$ under inclusion; and if $L$, $L^*$ are
consecutive convex subgroups in this chain with $L<L^*$, then
$L\triangleleft L^*$ and $L^*/L$ is isomorphic to a subgroup of $\RR$.
(See \cite[pp.\ 50--51]{F} or \cite[Lemmata 3.1.2, 3.1.3 and Corollary
4.2.5]{G99}.)

Write $w$ as in $(\ast)$ in the Introduction.
We may assume that $G=\langle h,z\rangle$.  Since $G$ is non-trivial and
finitely generated, it has a maximal
proper convex subgroup $L$, and $G/L$ is abelian.
Hence $h^{m}\in L$, where
$m=\sum m_i=A_w(1)$.
So either $A_w(1)=0$ and $1$ is a positive root of $A_w(X)$, or $h\in
L$.  We may assume the latter; thus  $L\neq\{ 1\}$ and $z\notin L$.

Let $Y:=\{h^{z^e},h^{z^{e+1}},\dots,
h^{z^{d-1}}\}$, where $e:=e_w$ and $d:=d_w$ are as defined in the Introduction.
Let $C^*$ be the smallest convex subgroup of $G$ containing $Y$.
Suppose that $n\geq d$ and $h^{z^i}\in C^*$ for $e\leq i<n$, but that $g\notin C^\ast$ where
$g=h^{z^n}$.
Let $D^\ast$, $D$ be respectively the smallest convex subgroup containing $g$ and the largest convex subgroup not containing $g$.
Then $C^*\leq D \triangleleft D^*$ and $D^\ast/D$ is abelian. We have
$w(h,z)^{z^{n-d}}=1\in D\lhd D^\ast$; but the element on the left is a product of conjugates of $h$ that lie in $D$ and powers of $g$ whose product is $g^{m'}$ where $m'=\sum_{d_j=d}m_j\neq0$.
Thus $g^{m'}\in D$. Since $D^\ast/D$ is torsion-free and $g\in D^\ast \setminus D$, this is a contradiction.
We conclude that $h^{z^n}\in C^*$ for all $n\geq e$.  A
similar argument shows that $h^{z^n}\in C^*$ for all $n \leq d$.
Therefore $C^*\triangleleft G$, and it follows that $C^*=L$.

Since $Y$ is finite, there is a largest convex subgroup $C$ of $G$ not containing $Y$.  Thus $C$ is the unique maximal convex subgroup strictly contained in $C^\ast$, and so $C\triangleleft G$ and $C^\ast/C$ is isomorphic to a subgroup of $\RR$.
By Hion's Lemma \cite[Theorem 1.5.1]{MR}, conjugation by $z$ must act on
$C^*/C$ as multiplication by a positive real number,
$\lambda$ say.  Our relation $(\ast)$ satisfied by $h$, $z$ holds for the
images $\bar h$, $\bar z$ modulo $C$.
In module notation this becomes
$$0=\bar h A_w(\bar z)= \bar h A_w(\lambda).$$
The theorem follows. \hfill $\Box$

\section{Proof of Theorem \ref{resnilptoorder}}

We begin with a remark and an elementary result.

\begin{remark} \label{rem}  {\em Suppose that $K\triangleleft G$, that $G/K$ has a bi-order $<_1$ and that $K$ has a bi-order $<_2$ with the property that $k_1<_2k_2$ and $g\in G$ implies $k_1^g\leq_2 k_2^g$ for all $k_1,k_2\in K$ and all $g\in G$.  For $g_1,g_2\in G$ write
$g_1<g_2$ if either $g_1K<_1g_2K$ or $g_1K=g_2K$ and $1<_2g_1^{-1}g_2$.  Then $<$ is a bi-order on $G$.}  \end{remark}

\begin{lem}  \label{vs} {\rm (cf.\ \cite[Proposition 2.2]{LRR})} Let $V$ be a finite-dimensional real vector space and $\Phi$ be an abelian group of linear maps all of whose eigenvalues are real and positive.  Then $V$ has a bi-order $<$ preserved by $\Phi$.\end{lem}

\begin{proof}  We argue by induction on $\dim V$. The result is clear if $V=\{0\}$.  The result is also clear if each element of $\Phi$ acts as multiplication by some real number.  Otherwise, some element
$\varphi$ of $\Phi$ has an eigenspace $W$ with $0<W<V$, and $W$ is a $\Phi$-invariant subspace since $\Phi$ is abelian.
Now
the eigenvalues of the maps on $W$, $V/W$ induced by the elements of $\Phi$ are real and positive, and therefore both $W$ and $V/W$ have suitable bi-orders by induction.  Hence so does $V$ by Remark \ref{rem}. \end{proof}

We now prove Theorem \ref{resnilptoorder}.  As usual we write $\ga_n(K)$ for the $n$th term of the lower central series of a group $K$.

\begin{proof} Suppose that $G$, $K$, $V$ and $\Phi$ are as in the statement of the theorem. For $\varphi\in \Phi$, let $\bar\varphi$ be the linear map induced on $V$ by $\varphi$.
For each $n$ let $T_n/\gamma_n(K)$ be the torsion subgroup of $K/\gamma_n(K)$ and
write $G_n=G/T_n$.  Since $\bigcap T_n=1$, the group $K$ embeds in the Cartesian product of the groups $K/T_n$.  However Cartesian products of bi-orderable groups are bi-orderable. So to prove the result, we may assume that $K$ is nilpotent, of class $c$, say.

If $c=1$ then Lemma \ref{vs} shows that $K$ has a bi-order $<$ with the property that $k_1<k_2$ if and only if $k_1\varphi<k_2\varphi$ for each $\varphi\in\Phi$, and so $G$ has a bi-order
with the desired property by Remark \ref{rem}.  For $c>1$ we may assume by induction that $G/L$ has a suitable bi-order, where $L/\gamma_c(K)$ is the torsion subgroup of $G/\gamma_c(K)$.  However $L$ is abelian (see \cite[5.2.19]{R}) and $L\otimes \RR= \gamma_c(K)\otimes\RR$.  The commutator map from the product of $c$ copies of $K$ to $\gamma_c(K)$ induces a $c$-linear map from the product of $c$ copies of $K/K'$ to $\gamma_c(K)$ and a surjective map from $\bigotimes_1^c (K/K')$ to $\gamma_c(K)$ (cf.\ \cite[5.2.5]{R}); this is a homomorphism of $\ZZ \Phi$-modules, where
$\Phi$ acts diagonally on the tensor power.  On tensoring with $\RR$, we obtain a surjective homomorphism of $\RR \Phi$-modules.  Since the eigenvalues of the maps induced by the elements of $\Phi$ on the tensor product are products of $c$ (not necessarily distinct) eigenvalues of elements $\varphi\in\Phi$, they are all positive.

By Lemma \ref {vs}, the group $L\otimes\RR=\gamma_c(K)\otimes\RR$ can be given a bi-order $<$ invariant under the maps $\varphi\otimes 1$. This gives a bi-order $<$ on $L$ such that $u_1\varphi<u_2\varphi$ if and only if $u_1<u_2$, for all $u_1,u_2\in L$ and all $\phi\in\Phi$.  Since $L$ is in the centre of $K$, by Remark \ref{rem} we can therefore combine this bi-order with the bi-order on $G/L$ to obtain a bi-order on $G$, as required.  \end{proof}

\section{Proof of Theorem \ref{B}}

We shall write $\Ntf$ for the class of torsion-free nilpotent groups.  We recall that a group is said to have {\em finite Pr\"ufer rank} if there is an integer $r$ such that every finitely generated subgroup can be generated by $r$ elements; its {\em Pr\"ufer rank} is the smallest such $r$ (see \cite[Exercises 14.1, 1--4]{R}).  Clearly
no subgroup of such a group has an infinite abelian quotient of finite exponent.  Moreover if $G$ is a soluble such group then $G$ has a finite subnormal series in which each factor is either infinite cyclic or a torsion group; the number of infinite cyclic factors is the {\em torsion-free rank} of $G$.
We recall some elementary facts.

\begin{lem}\label{bigsubgroups} \begin{enumerate} \item[\rm (a)]   Let $G$ be a group.
\begin{enumerate}\item[\rm(i)] If $G/\gamma_2(G)$ can be generated by $r$ elements then so can
$G/\gamma_n(G)$ for all $n\geq 2$.
\item[\rm(ii)] If $G/\gamma_2(G)$ has finite Pr\"ufer rank then so does $G/\gamma_n(G)$ for all $n\geq 2$. \end{enumerate}
\item[\rm(b)] Suppose that $H$ has finite Pr\"ufer rank and $L$ is a subgroup such that $|H\colon LH'|$ is finite. Then $|H\colon L\gamma_{n+1}(H)|$ is finite for all $n>0$.

 In particular, if in addition $H$ is nilpotent, then $|H\colon L|$ is finite. \end{enumerate}  \end{lem}

\begin{proof} (a)(i)  It suffices to prove that if $G=H\gamma_n(G)$ for some subgroup $H$ and integer $n\geq2$ then $G=H\gamma_{n+1}(G)$.  Fix $H,n$ and write bars for images of subgroups modulo $\gamma_{n+1}(G)$.  Since $\overline{\gamma_n(G)}$ is central we have $\overline G '=\overline H '$, and hence $G'\leq H\gamma_{n+1}(G)$ and $G=HG'=H\gamma_{n+1}(G)$.

(ii)  Since the class of groups of finite Pr\"ufer rank is closed with respect to extensions it suffices to prove that if $n\geq2$ and $G/\gamma_n(G)$ has finite Pr\"ufer rank then so has
$\gamma_n(G)/\gamma_{n+1}(G)$.  This group is
the image of $\gamma_{n-1}(G)/\gamma_n(G)\otimes G/\gamma_2(G)$ under the map induced by the commutator map from $\gamma_{n-1}(G)\times G$ to $\gamma_n(G)/\gamma_{n+1}(G)$; see \cite[pp.\ 126--127]{R}.  However a tensor product of an abelian $r$-generator group and an abelian $s$-generator group can be generated by $rs$ elements, and it follows immediately
that a tensor product of two abelian groups of finite Pr\"ufer rank has itself finite Pr\"ufer rank.  The result follows.

(b)  It suffices to assume that $|H\colon L\gamma_n(H)|$
is finite for some $n\geq2$ and prove that $|H\colon L\gamma_{n+1}(H)|$ is finite; in proving this we can also assume that $\gamma_{n+1}(H)=1$.

Let $m=|H\colon LH'|$.  The $n$-fold commutator map $(x_1,\dots,x_n)\mapsto [x_1,\dots,x_n]$
from $H\times\cdots\times H$ to $\gamma_n(G)$ is a homomorphism in each of its $n$ variables and it induces an $n$-linear map from $H/H'\times\cdots\times H/H'$ to $\gamma_n(H)$.  It follows that $[x_1,\dots,x_n]^{m^n}=[x_1^m,\dots,x_n^m]\in L$ for all
$x_1,\dots,x_n\in H$.  Thus $\gamma_n(H)/(L\cap \gamma_n(H))$ is an abelian group of finite exponent. But $H$ has finite Pr\"ufer rank, and so $\gamma_n(H)/(L\cap\gamma_n(H))$ is finite.
Therefore
$$|H\colon L|=|H\colon L\gamma_n(H)|\,|L\gamma_n(H)\colon L|=|H\colon L\gamma_n(H)|\,|\gamma_n(H)\colon(L\cap\gamma_n(H))|$$
is finite,  as required. \end{proof}

\begin{remark}\label{tsec}
{\rm
Note that the hypothesis that $|H\colon LH'|$ is finite in (b) above holds if $H$ is obtained from $L$ by adjunction of solutions of finitely many one-variable group equations over $K$ with non-zero weight in the variable.}
\end{remark}

\begin{prop}\label{laboursaving} Let $G=K\rtimes\langle t\rangle$ and let $M$ be a subgroup such that $M^t\leq M$ and $K=\bigcup_{n\in\ZZ}M^{t^n}$.  Suppose that \begin{enumerate}
\item[\rm(i)]  $M^\ab$ has finite Pr\"ufer rank, and
\item[\rm(ii)]  $M$ is generated by $M^t$ and finitely many elements satisfying one-variable
group equations over $M^t$ with non-zero weight in the variable. \end {enumerate}
If $M$ is residually $\Ntf$ then so is $K$. \end{prop}

\begin{proof}  By Lemma \ref{bigsubgroups} (a)(ii), all nilpotent images of $M$ have finite Pr\"ufer rank.

For each $c\in\NN$ let $R_c/\gamma_{c+1}(M)$ be the torsion subgroup of $M/\gamma_{c+1}(M)$
and $S_c/\gamma_{c+1}(K)$ the torsion subgroup of $K/\gamma_{c+1}(K)$.  Then $R_c^t/\gamma_{c+1}(M^t)$ is the torsion subgroup of $M^t/\gamma_{c+1}(M^t)$ and $\bigcup_{n\geq0} R_c^{t^{-n}}=S_c.$  Moreover
 $M^t/R_c^t\cong M/R_c$;
in particular, these two torsion-free groups of finite Pr\"ufer rank have the same torsion-free rank.
By hypothesis (ii) and the remark, the index $|M/M':M^tM'/M'|$ is finite; therefore from Lemma  \ref{bigsubgroups}(b) we deduce that the subgroup $M^tR_c/R_c$ has finite index in $M/R_c$; therefore these groups too have the same torsion-free rank.  Thus $M^t/R_c^t$ has the same torsion-free rank as its homomorphic image $M^t/(R_c\cap M^t)$, and is torsion-free. It follows that $R_c\cap M^t= R_c^t$ and $R_c^{t^{-1}}\cap M=R_c$.
Conjugating repeatedly by powers of $t^{-1}$
we now find that $R_c^{t^{-n}}\cap M=R_c$ for all $n>0$.
Therefore
$$S_c\cap M=\bigg( \bigcup_n R_c^{t^{-n}}\bigg) \cap M
=\bigcup_n \big(R_c^{t^{-n}}\cap M\big)=R_c.$$
This holds for all $c$, and hence
$$\bigg(\bigcap_c S_c\bigg) \cap M=\bigcap_c R_c=1,$$
since $M$ is residually $\Ntf$.
Conjugating by $t^{-r}$ we now find that $(\bigcap S_c)\cap M^{t^{-r}}=\{1\}$ for each $r\geq0$.  Since $K=\bigcup M^{t^{-r}}$ it follows that $\bigcap_{c>0}S_c=\{1\}$, and that
$K$ is residually $\Ntf$, as required. \end{proof}

We now prove Theorem \ref{B}. Let $G$ have presentation $\langle x,t\mid w\rangle$, where $w$ is a monic word.
We may replace $w$ by $w^{t^{-e_w}}$ where $e_w$ is defined as in the Introduction.
Let $A_w(X)$ have degree $d$ and suppose that all of its roots are real and positive.

Write $x_i=x^{t^i}$ for each $i\in\ZZ$.
Thus $x_d$ is a product of the elements $x_i^{\pm 1}$ with $0\leq i<d$.
It will
suffice to prove that the group $K:=\langle x_i\mid i\in\ZZ\rangle$ is residually $\Ntf$ since then the result follows from Theorem \ref{resnilptoorder}.

By the proof of the Freiheitssatz (see \cite[p.\ 199]{LS}), the subgroup $F$ of $G$ generated by $x_0,\dots,x_{d-1}$ is free on these generators: in particular, it is residually $\Ntf$.
Since $F$ is finitely generated, its nilpotent images certainly have finite Pr\"ufer rank.  Therefore the hypotheses of Proposition \ref{laboursaving} hold, and we conclude that
$K$ is residually $\Ntf$, as required.   \hfill $\Box$

\section{Extraction of roots in groups}

To prove Theorem \ref{nm}, we need to develop G. Baumslag's theory of roots in groups, as described in \cite{B1} and \cite{B2}.

In \cite[Chapter VI]{B1} a class $\mathcal{D}_{\omega}$ of groups is introduced for each set $\omega$ of primes.   We write $\mathcal{D}$ for the class corresponding to the set of all primes.
Thus a group $G$ is in $\mathcal{D}$ if it satisfies the following four conditions:
\begin{itemize}
\item[(1)] if $g^r=h^r$, where $g$, $h\in G$ and $r\in\ZZ\setminus\{0\}$, then $g=h$;
\item[(2)] if $g\in G$ has no $p$th root in $G$ for some prime $p$, then the centralizer $\C_G(g)$ is isomorphic to a
subgroup of the additive group of $\QQ$;
\item[(3)] if $g\in G$ has no $p$th root in $G$ for some prime $p$, then  $\C_G(g^r)=\C_G(g)$ for all
integers $r\ne 0$;
\item[(4)] if $g\in G$ has no $p$th root in $G$ for some prime $p$, and $f^{-1}g^rf=g^s$, where $f\in
G$ and $r,s\in\ZZ\setminus\{0\}$, then $r=s$.
\end{itemize}
By (1), every group in $\mathcal{D}$ is torsion-free.
\begin{proposition} \label{fpinD}Suppose that $G\in \mathcal{D}$ and $u\in G\setminus\{1\}$ is such that $\C_G(u)=\langle u\rangle$. Let $m$ be a non-zero integer. Then the free product with amalgamation $H:=G*_{(u=x^m)}\langle x\rangle$ is also in $\mathcal{D}$.
\end{proposition}
\proof First note that, if $p$ is a prime, then since $G$ is torsion-free and $C_G(u)=\langle u\rangle$, it follows that $u$ has no $p$th root in $G$.
We can assume that $m\geq1$.  If $m=1$ then $H=G$, so we can assume that $m>1$. Let $P$ be a group (written multiplicatively) isomorphic to the additive group of $\mathbb{Q}$, and take an embedding of $\langle u\rangle$ into $P$. By \cite[Theorem 29.1]{B1}, the group $K:=G*_{\langle u\rangle}P$ is in the class $\mathcal{D}$. Moreover, we can embed $H$ in $K$ by mapping $x$ to the $m$th root of $u$ in $P$. Since condition (1) is clearly preserved by subgroups, $H$ satisfies (1).
Before we proceed to establish (2)--(4), two remarks are needed.

Suppose that $h\in H\leq K$.
Then $h$ has a reduced decomposition in the amalgamated free product $K=G\ast_{\langle u\rangle}P$; say $h=k_1\ldots k_n$, where the factors $k_i$ come alternately from $G$ and $P$, and none of them is in the amalgamated subgroup $\langle x^m\rangle$, unless $n=1$. Similarly, we can write a reduced decomposition of $h$ in $H$, say $h=h_1\ldots h_l$, with factors $h_i$ alternately from $G$ and $\langle x\rangle$. But this is a reduced decomposition of $h$ in $K$. It follows from the Reduced Form Theorem (\cite[Chapter 1, Theorem 26]{DC}) that $n=l$ and $k_i\in \langle x^m\rangle h_i\langle x^m\rangle$ for $1\leq i \leq n$. Hence $h=k_1\ldots k_n$ is also a reduced decomposition of $h$ in $H$.

Suppose that $h^{-1}x^rh=x^s$, where $h\in H$ and $r,s\in\ZZ$. We claim that $r=s$.
For write $h=h_1\ldots h_n$ in reduced form in $H$, and use induction on $n$.
First suppose that $n=1$.  Clearly $r=s$ if $h_1\in \langle x\rangle$. Otherwise, $h_1\in G\setminus\langle x\rangle$, and by the Reduced Form Theorem, $x^r\in \langle x^m\rangle$; so $x^s\in G$, and hence $x^s\in \langle x^m\rangle$. Thus $r=mr'$, $s=ms'$ for some $r'$, $s'\in\ZZ$, and $h_1^{-1}u^{r'}h_1=u^{s'}$. Since $G\in \mathcal{D}$ and $u$ has no $p$th root in $G$, we have $r'=s'$, and hence $r=s$.
Now suppose that $n>1$. Then $h_1^{-1}x^rh_1\in \langle x\rangle$. For suppose not; then $h_1\in G\setminus \langle x\rangle$, so depending on whether or not $m$ divides $r$, either $h_n^{-1}\ldots h_1^{-1}x^rh_1\ldots h_n$ is a reduced word of length $2n+1$ or $h_n^{-1}\ldots h_2^{-1}(h_1^{-1}x^rh_1)h_2\ldots h_n$ is a reduced word of length $2n-1>1$ representing $x^b$, an element of the free factor $\langle x\rangle$ of $H$. This contradicts the Reduced Form Theorem, establishing that
$h_1^{-1}x^rh_1\in \langle x\rangle$. By the case $n=1$, we now have $h_1^{-1}x^rh_1=x^r$,  and so $h^{-1}x^rh= h_n^{-1}\ldots h_2^{-1}x^rh_2\ldots h_n=x^s$, and by induction $r=s$.

By \cite[Lemma 28.1]{B1}, if $r\neq 0$ then $\C_K(x^r)=P$, and so $\C_H(x^r)=P\cap H=\langle x\rangle$.  Suppose that $h\in H$ and $h$ has no $p$th root in $H$. In verifying that $H$ satisfies (2)--(4), we can assume that $h$ is cyclically reduced. If $h$ is conjugate in $H$ to a power of $x$, then (2)--(4) follow so we can assume that this is not the case.  Suppose then that $h$ has a $p$th root in $K$; this $p$th root is conjugate in $K$ to a cyclically reduced element, say $k$, and so $h$ is conjugate in $K$ to $k^p$. Let $k=k_1\ldots k_n$ in reduced form. If $n>1$ then $(k_1\ldots k_n)^p$ is a cyclically reduced decomposition of $k^p$ in $K$ of length $np$. By the Conjugacy Theorem (\cite[Theorem 4.6]{MKS}), $h'=(k_1\ldots k_n)^p$, where $h'$ is obtained from $h$ by cyclic permutation and conjugating by an element of $\langle x^m\rangle$. In particular, $h'$ is conjugate to $h$ in $H$. As noted above, this is a decomposition of $h'$ in $H$, and hence $k\in H$; therefore both $h'$, $h$ are $p$th powers in $H$, a contradiction. If $n=1$, then by the Conjugacy Theorem, $h$ is conjugate in $G$ to $k_1^p$. If $k_1\in P$, then $k_1^p\in \langle x^m\rangle$, which is impossible, so $k_1\in G$, and hence $h$ is a $p$th power in $G$,
which is again a contradiction. Hence $h$ has no $p$th root in $K$.  Thus $H$ satisfies (3) and (4) since $K\in \mathcal{D}$.

It remains to verify that $\C_H(h)$ is isomorphic to a subgroup of $\QQ$. If $h$ has length at least 2 relative to $H$, then by \cite[Lemma 28.6]{B1}, $\C_K(h)$ is cyclic, hence so is $\C_H(h)$, and (2) follows since $H$ is torsion-free. Finally suppose that $h\in G$. Since $h$ has no $p$th root in $K$, it is not conjugate in $K$ to an element of the amalgamated subgroup $\langle x^m\rangle$; so by \cite[Lemma 28.2]{B1} we conclude that $\C_K(h)=\C_G(h)$ and therefore $\C_H(h)=\C_G(h)$.  Since $G\in \mathcal{D}$ and $h$ has no $p$th root in $G$, it follows that $\C_G(h)$ is isomorphic to a subgroup of $\QQ$.
\endproof

\medskip\noindent{\bf Definitions.}  \ We call an element $g$ of a group $G$ {\em indivisible} if $g$ is not a proper power of any element of $G$.  We write $\mathcal{D}_0$ for
the class  of all groups $G$ in $\mathcal{D}$ satisfying
\begin{enumerate}
\item[\rm (5)] if $g\in G$ is indivisible in $G$ then $\C_G(g)=\langle g\rangle$.
\end{enumerate}
\smallskip

Note that free groups belong to $\mathcal{D}_0$.  It is shown that they are in $\mathcal{D}$ in \cite[Corollary 35.7]{B1}, and property (5) is well known (see, for example, \cite[Chapter 1, Proposition 12]{DC}).

\begin{corollary} \label{fpinD0}Suppose that $G\in \mathcal{D}_0$ and $ u$ is an indivisible element of $G$. Let $m\in\ZZ\setminus\{0\}$. Then $H:=G*_{(u=x^m)}\langle x\rangle\in \mathcal{D}_0$.
\end{corollary}

\proof We can assume that $m>1$.  In view of Proposition \ref{fpinD}, we only need to show that $H$ satisfies (5).    Let $K$ be defined as in the proof of Proposition \ref{fpinD}.   Fix an indivisible element $h$, which we can assume to be cyclically reduced.  If $h$ is conjugate in $H$ to $x^{\pm 1}$
then, as noted previously, $\C_H(x^{\pm 1})=\langle x\rangle$, hence $\C_H(h)=\langle h\rangle$. Otherwise, either $h\in G$ or $h$ has length at least $2$ relative to the free product decomposition of $H$. From the proof of Proposition \ref{fpinD}, $h$ is not a $p$th power in $K$ for any prime $p$; that is, $h$ is indivisible in $K$.

If $h$ has length at least $2$ relative to $H$, then by Lemma 28.6 (and its proof) in  \cite{B1}, $\C_K(h)$ is cyclic, with a generator $h^*$ such that $h=(h^*)^r$ for some $r>0$. Since $h$ is indivisible in $K$, we must have $r=1$ and $\C_H(h)=\langle h\rangle$ as required. Finally suppose that $h\in G$. Since $h$ is indivisible in $K$, it is not conjugate in $K$ to an element of the amalgamated subgroup $\langle x^m\rangle$; so $\C_K(h)=\C_G(h)$ by \cite[Lemma 28.2]{B1}, and hence $\C_H(h)=\C_G(h)$.  Since $G\in \mathcal{D}_0$ and $h$ is indivisible in $G$, we have $\C_G(h)=\langle h\rangle$.
\endproof

The next result is a mild generalization of  \cite[Corollary 2]{B2} and the method of proof is the same.

\begin{proposition} \label{fpresp}Suppose that $G$ is finitely generated and residually a finite $p$-group, where $p$ is a prime.  Let $u$ be an element of $G$ of infinite order and assume that $\C_G(u)=\langle u\rangle$. Then for every $n\in\ZZ_{\geq 0}$, the group $H:=G*_{(u=x^{p^n})}\langle x\rangle$ is residually a finite $p$-group.
\end{proposition}
\proof  Let $L$ be the kernel of the homomorphism
$H\to \langle  x\rangle /\langle  x^{p^n}\rangle$ sending $x$ to $x\langle  x^{p^n}\rangle$ and $G$ to $1$.
Each element of $H$ can be written in the form $x^ky$, where $y$ is a product of elements of the groups $x^{-r}Gx^r$ for $r\in \ZZ$, and $0\leq k\leq p^n-1$. Hence $L$ is generated by $\bigcup_{r\in \ZZ}x^{-r}Gx^r=\bigcup_{r=0}^{p^n-1}x^{-r}Gx^r$. It follows from the Reduced Form Theorem for free products with amalgamation that $L$ is the free product of the groups $x^{-r}Gx^r$ for $0\leq r\leq p^n-1$, amalgamating their common subgroup $\langle u\rangle$. By the argument for Lemma 1 in \cite{B2}, $L$ is residually a finite $p$-group; since $L$ is finitely generated and normal of index $p^n$ in $H$, the group $H$ is residually a finite $p$-group.
\endproof

Recall (from \cite{B2}) that two groups $A$, $B$ are said to have the same lower central sequence if there are isomorphisms $\psi_n\colon A/\gamma_n(A)\to B/\gamma_n(B)$, such that $\psi_n$ induces $\psi_{n-1}$, for all $n\geq 2$; and a group $G$ is termed \emph{parafree of rank} $r$ if $G$ is residually nilpotent and $G$ has the same lower central sequence as a free group of rank $r$.

\begin{lem} \label{fppfr} Let $G$ be a group, $u$ be an element of infinite order in $G$ and $H:=G*_{(u=x^m)}\langle x\rangle$, where $m\ne 0$. Suppose that $G$ has the same lower central sequence as a free group of rank $r$,  for some $r\in \ZZ_{>0}$, and $H/\gamma_2(H)$ is $r$-generated. Then $H$ has the same lower central sequence as a free group of rank $r$.
\end{lem}
\proof Arguing as in the first paragraph of the proof of \cite[Proposition 2]{B2} (top of p.\ 312), and using Lemma \ref{bigsubgroups} (a) we obtain the result from \cite[Proposition 1]{B2}.
\endproof
We shall also need the following simple observation on free products with amalgamation.

\begin{lem} \label{vind0}Let $G=A*_HB$ be a free product with amalgamation and suppose that $b$ is an indivisible element of $B$ which is not conjugate in $B$ to an element of $H$. Then $b$ is an indivisible element of $G$.
\end{lem}
\proof Suppose that $b=v^n$, where $v\in G$ and $n>1$. Then $v=z^{-1}wz$ for some $z\in G$ and cyclically reduced $w$. By the Conjugacy Theorem for free products with amalgamation (see \cite [Theorem 4.6]{MKS}),
$b$ is conjugate to $w^n$ in $B$, and $w^n\in B\setminus H$. It follows that $w\in B$.  Hence $b$ is a proper power in $B$, which is a contradiction.
\endproof

\section{Proof of Theorem \ref{nm}}

Suppose that $G=\langle x,t\mid w\rangle$, for some principal word $w=w(x,t)$.
Conjugating $w$ by a suitable power of $t$ and making a cyclic permutation if necessary, we can assume that $w(x,t)=u(x,x^t,\ldots,x^{t^{d-1}})x^{-mt^d}$ for some $d$ and $m\ne 0$, where $u(x_0,\ldots,x_{d-1})$ is a word representing an element of the free group on $x_0,\ldots,x_{d-1}$, such that there is an occurrence of $x_0$ or $x_0^{-1}$  in $u(x_0,\ldots,x_{d-1})$. Then $G$ is the HNN extension
\[ G=\langle t,x_0,\ldots,x_d\mid u(x_0,x_1,\ldots,x_{d-1})=x_d^m, x_i^t=x_{i+1}\ (0\leq i\leq d-1)\rangle \]
with base a one-relator group and free amalgamated subgroups (by the Freiheitssatz).
We shall prove that, under certain conditions, the normal subgroup $K:=\langle x^{t^n}\mid n\in\ZZ\rangle$ of $G$ generated by $x$ is residually $\Ntf$. We recall that $\Ntf$ is the class of torsion-free nilpotent groups.

The group $K$ is the kernel of the retraction $G\to\langle t\rangle$ and has a well-known structure as a tree product, the tree having vertex set $\ZZ$ with edges joining $i$ and $i+1$, for all $i\in\ZZ$. To describe this structure, take a countable set of symbols $\{ x_i\mid i\in\ZZ\}$ and for $i\in\ZZ$, let $F_i$ be the free group on $x_i,\dots,x_{i+d-1}$.  Let $u(x_0,\ldots,x_{d-1})$ be a reduced word representing an element of $F_0$, and assume that there is at least one occurrence of $x_0^{\pm 1}$ in $u(x_0,\dots,x_{d-1})$.

For $i\in \ZZ$, define $V_i$ by
$$V_i:=F_i\ast_{(u(x_i,\dots,x_{i+d-1})=x_{i+d}^m)} \;\langle x_{i+d}\rangle.$$
By the Freiheitssatz, $x_{i+1},\ldots ,x_{i+d}$ freely generate a subgroup of $V_i$.
Hence, for any $i,j\in \ZZ$ with $j\geq i$, we can define $V_{i,j}$ inductively by

$$V_{i,i}:=V_i \quad \hbox{and} \quad V_{i,j+1}:= V_{i,j}\; \ast_{\langle x_{j
+1},\dots,x_{j+d}\rangle}\; V_{j+1}.\eqno(\dag)$$
Note that
$$
V_{i,j+1}=V_{i,j}\;\ast_{(u(x_{j+1},\dots x_{j+d})=x_{j+d+1}^m)} \; \langle x_{j+d+1} \rangle.\eqno(\dagger\dag)
$$
Thus $V_{i,j}$ has a presentation with $j+d+1$ generators and $j+1$ relations.

The tree product structure of $K$ can be described by observing that there is an isomorphism $\bigcup\{V_{i,j}\mid i,j\in \ZZ,\;i\leq j\}\to K$, sending $x_n$ to $x^{t^n}$ for $n\in\ZZ$. We begin by showing that, under certain conditions, $V_{i,j}$ is a parafree group.

Let $a_k$ be the weight of $x_k$ in the word $u(x_0,\ldots,x_{d-1})x_d^{-m}$ for $0\leq k\leq d$. Let $R_j$ be the $(j+1)\times (d+j+1)$ integer matrix
\[ \begin{pmatrix}
a_0&a_1&a_2&\ldots&a_d&0&0&\ldots&0\\
0&a_0&a_1&\ldots&a_{d-1}&a_d&0&\ldots&0\\
0&0&a_0&\ldots&a_{d-2}&a_{d-1}&a_d&\ldots&0\\
\vdots&&&&&&&&\vdots\\
0&0&\ldots&\qquad a_0&&\ldots&&a_{d-1}&a_d
\end{pmatrix}.\]
Further, let $R_j(a)$ be the result of replacing the bottom right-hand entry of $R_j$ by $-a$. Note that $a_d=-m$. Recall the definition of Smith normal form \cite[Section 10.5, Theorem 4]{PC}. We are interested in the situation where $R_j$ and $R_j(a)$ have Smith normal form over $\ZZ$ with entries $1$ on the diagonal, that is, the identity matrix $I_{j+1}$ followed by $d$ columns of zeros. This is equivalent to the condition that  the greatest common divisor of the $(j+1)\times (j+1)$ minors is $1$. Suppose that $a$ divides $m$.  Each $(j+1)\times (j+1)$ submatrix of $R_j(a)$ including the last column has determinant dividing that of the corresponding submatrix in $R_j$. The other $(j+1)\times (j+1)$ submatrices are also submatrices of $R_j$. It follows that the greatest common divisor of the $(j+1)\times (j+1)$ minors of $R_j(a)$ divides the greatest common divisor of the $(j+1)\times (j+1)$ minors of $R_j$. Hence, if $R_j$ has Smith normal form with entries $1$ on the diagonal, then so does $R_j(a)$.

The following is the main result which will enable us to prove that $K$ is residually nilpotent. Assume that $m>0$.

\begin{proposition} Let $a_k$ be the weight of $x_k$ in $u(x_0,\ldots,x_{d-1})x_d^{-m}$ for $0\leq k\leq d$. Assume that
\begin{itemize}
\item[\rm (a)]
$u(x_0,\dots,x_{d-1})$ is indivisible in $F_0$,
\item[\rm (b)]
$u(x_1,\dots,x_{d})$ is indivisible in $V_0$ and  not conjugate in $V_0$ to an element of $F_0$, and
\item[\rm (c)]
for all $j\geq 0$, the matrix $R_j$ has Smith normal form with entries $1$ on the diagonal.
\end{itemize}
Then for all $j\geq i$, $V_{i,j}$ is parafree of rank $d$.
\end{proposition}
\proof Since $V_{i,j}$ is isomorphic to $V_{0,j-i}$, it suffices to show that $V_{0,j}$ is parafree of rank $d$ for all $j\geq 0$. We show by induction on $j$ that $V_{0,j}$ is parafree of rank $d$ and $V_{0,j}\in\mathcal{D}_0$. Suppose that $m=p_1^{n_1}\ldots p_l^{n_l}$ is the decomposition of $m$ as a product of powers of distinct primes. Take new symbols $y_1,\ldots, y_l$, put $q_i=p_i^{n_i}$ and define
\begin{gather*} U_0:=F_0, \quad U_1:=U_0*_{(u(x_0,\ldots, x_{d-1})=y_1^{q_1})}\langle y_1\rangle, \\
U_2:=U_1*_{(y_1=y_2^{q_2})}\langle y_2\rangle,\quad \ldots,
\quad U_l:=U_{l-1}*_{(y_{l-1}=y_l^{q_l})}\langle y_l\rangle.
\end{gather*}
Thus $U_l\cong V_0=V_{0,0}$. Now $U_i/U'_i$ is generated by $x_0U'_i,\ldots,x_{d-1}U_i',y_iU_i'$ subject to the relation $xS_0(q_1\ldots q_i)=0$ (in additive notation), where $S_0(q_1\dots q_i)$ is the transpose of  $R_0(q_1\dots q_i)$ and $x=(x_0U_i',\dots, x_{d-1}U_i', y_iU_i')$; so by assumption (c) and the remarks preceding the proposition, $U_i/U_i'$ is free abelian of rank $d$. By induction on $i$ and Lemma \ref{fppfr}, we conclude that $U_i$ has the same lower central sequence as $F_0$.
We show by induction that $U_i$ is parafree of rank $d$ and $U_i\in\mathcal{D}_0$.
This is clear if $i=0$. Assume that $i>0$ and that $U_{i-1}$ is parafree of rank $d$ and belongs to $\mathcal{D}_0$. By assumption,
$u(x_0,\ldots ,x_{d-1})$ is indivisible in $U_0$, and $y_{i-1}$ is indivisible in $U_{i-1}$ for $i>1$ by Lemma \ref{vind0}. By Corollary \ref{fpinD0} we have $U_i\in\mathcal{D}_0$. By assumption, $U_{i-1}$ is residually $\mathcal{N}_{\rm tf}$, and hence is residually a finite $p_i$-group. By Proposition \ref{fpresp} the subgroup $U_i$ is residually a finite $p_i$-group, hence is parafree of rank $d$.
Thus $V_0\cong U_l$ is parafree of rank $d$ and is in $\mathcal{D}_0$.

Assume that $V_{0,j}$ is parafree of rank $d$ and belongs to $\mathcal{D}_0$. Define
\begin{gather*} W_0:=V_{0,j},\quad W_1:=W_0*_{(u(x_{j+1},\ldots, x_{j+d})=y_1^{q_1})}\langle y_1\rangle,\\
W_2:=W_1*_{(y_1=y_2^{q_2})}\langle y_2\rangle,\quad \ldots,
\quad
 W_l:=W_{l-1}*_{(y_{l-1}=y_l^{q_l})}\langle y_l\rangle.
\end{gather*}
By a similar inductive argument to that just given, $W_i$ is parafree of rank $d$ and belongs to $\mathcal{D}_0$ for $0\leq i\leq l$.  The main points are, first, that $W_i/W_i'$ is generated by
$x_0W_i',\ldots,x_{j+d}W_i',y_iW_i'$ subject to the relations
$xS_{j+1}(q_1\ldots q_i)=0$, where $x= (x_0W_i', \dots, x_{j+d}W_i', y_iW_i')$ and $S_{j+1}(q_1\ldots q_i)$ is the transpose of $R_{j+1}(q_1\ldots q_i)$. Secondly, if $j\geq 1$, then $V_0\cong V_j$ via the mapping $x_k\mapsto x_{k+j}\ (0\leq k\leq d)$, so by assumption (b) the element $u(x_{j+1},\ldots,x_{j+d})$ is indivisible in $V_j$ and is not conjugate in $V_j$ to an element of $F_j$. Using $(\dagger)$, we conclude that $W_0=V_{0,j-1}*_{F_j}V_j$,
and hence by Lemma \ref{vind0} that $u(x_{j+1},\ldots, x_{j+d})$ is indivisible in $W_{0}$.

Since $V_{0,j+1}\cong W_l$, this finishes the induction.
\endproof

We give a simple criterion for condition (c) to hold.

\begin{corollary} \label{0.8} Suppose that $m=p_1^{n_1}\ldots p_l^{n_l}$ is the decomposition of $m$ into prime powers and let $a_k$ be the weight of $x_k$ in the word $u(x_0,\ldots,x_{d-1})x_d^{-m}$ for $0\leq k\leq d$. Assume that
\begin{itemize}
\item[\rm (a)\phantom{$'$}]
$u(x_0,\dots,x_{d-1})$ is indivisible in $F_0$,
\item[\rm (b)\phantom{$'$}] $u(x_1,\dots,x_{d})$ is indivisible in $V_0$ and not conjugate in $V_0$ to an element of $F_0$, and
\item[\rm (c)$'$] for $1\leq i\leq l$, there exists $a_{r_i}$ which is not divisible by $p_i$.
\end{itemize}
Then for all $j\geq i$, $V_{i,j}$ is parafree of rank $d$.
\end{corollary}
\proof We need to show that, if $j\geq 0$, the greatest common divisor of the set of all $(j+1)\times (j+1)$ minors of $R_j$ is $1$. Choose $r_i$ to be minimal such that $p_i$ does not divide $a_{r_i}$. Taking the $j+1$ columns of $R_j$ starting with column $r_i+1$ we obtain a matrix which has $a_{r_i}$ along the main diagonal and all entries below the main diagonal divisible by $p_i$. Hence the determinant of this submatrix is congruent modulo $p_i$ to $a_{r_i}^{j+1}$, so is not divisible by $p_i$.

The submatrix consisting of the last $j+1$ columns of $R_j$ is lower triangular, and has determinant $(-m)^{j+1}$, which is divisible only by primes in $\{ p_1,\ldots,p_l\}$. Hence the greatest common divisor of the $(j+1)\times (j+1)$ minors of $R_j$ is indeed $1$.
\endproof

Since verification of hypothesis (b) can be quite tedious even in
simple cases, we give the following version that provides an easily applied
algorithm.

\begin{corollary}\label{0.9} Suppose that $a_0,\ldots ,a_{d-1}$ have greatest common divisor $1$ and $m$ does not divide $a_{d-1}$. Then for all $j\geq i$, $V_{i,j}$ is parafree of rank $d$.
\end{corollary}
\proof We need to verify that conditions (a) and (b) in Corollary \ref{0.8} hold.  Let $y_i=x_iV_0'$ for $0\leq i\leq d$. The obvious map from $V_0$ to the cyclic group of order $m$ (cf.\ the proof of Proposition \ref{fpresp}) induces a map $\theta$ on $V_0/V_0'$, sending $y_d$ to a generator and $y_i$ to $1$ for $0\leq i\leq d-1$. By the proof of Lemma \ref{fppfr}, $y_0,\ldots,y_{d-1}$ form a basis for a free abelian subgroup, say $A$, of $V_0/V_0'$; writing this group additively, we see that $a_0y_0+\cdots + a_{d-1}y_{d-1}$ is not a proper multiple in $A$, so $u(x_0,\dots,x_{d-1})$ is indivisible in $F_0$. Hence (a) holds. To deduce (b), suppose first that $u(x_1,\dots,x_{d})$ is a proper power in $V_0$; then there is an equation
\[ a_0y_1+\cdots + a_{d-1}y_d=k(b_0y_0+\cdots + b_dy_d)\]
where $k>1$. Applying $\theta$, we find that $a_{d-1}=kb_d+mr$ for some integer $r$. This gives
\[ r a_0 y_0 + (a_0+r a_1)y_1+ \cdots +(a_{d-2}+ra_{d-1})y_{d-1}=k(b_0y_0+\cdots + b_{d-1}y_{d-1}).\]
Let $p$ be a prime divisor of $k$; so $p$ divides all coefficients on the left-hand side.
If $p|r$, then $p$ divides $a_0,\dots,a_{d-2}$.
Since $a_{d-1}=kb_d+mr$ and $p|r,k$, then $p$ also divides $a_{d-1}$, a contradiction.
If $(p,r)=1$ then $p$ divides $a_0$ ($p$ divides the $y_0$ coefficient), and hence $a_1,\dots,a_{d-1}$, a contradiction. Thus $u(x_1,\dots,x_d)$ is indivisible in $V_0$.
Finally, if $u(x_1,\ldots,x_d)$ is conjugate in $V_0$ to an element of $F_0$, then $(a_0y_1+\cdots + a_{d-1}y_d)\theta=0$, so $m$ divides $a_{d-1}$, and this is a contradiction.
\endproof

The following theorem implies Theorem \ref{nm} {\it a fortiori}.

\begin{theorem} \label{nonmon}
Let $G$ be the group with presentation $\langle x,t\mid w(x,t)\rangle$,  where $w(x,t)$ can be written as $u(x_0,\dots,x_{d-1})x_{d}^{-m}$ with $m>0$ and $x_i:=x^{t^i}$ for $i\in \ZZ$. If $u$ satisfies  the hypotheses of either Corollary {\rm \ref{0.8}} or Corollary {\rm\ref{0.9}} and all roots of $A_w(X)$ are positive, then the normal subgroup of $G$ generated by $x$ is residually $($torsion-free nilpotent\/$)$, and consequently $G$ is bi-orderable.
\end{theorem}

\proof

Note that $a_k$, the weight of $x_k$ in $u(x_0,\dots,x_{d-1})x_d^{-m}$, is the coefficient of $X^k$ in $A_w(X)$.

Recall that $K$ is the normal subgroup of $G$ generated by $x$. We have identified $x_i$ with $x^{t^i}$, so that $K=\bigcup\{V_{i,j}\mid i,j\in \ZZ,\;i\leq j\}$. By Theorem \ref{resnilptoorder}, it is enough to prove that $K$ is residually $\Ntf$.
Let $K_+:=\bigcup\{V_{0,j}\mid j\in \ZZ_{\geq 0}\}$.  
For each $j$, the group $V_{0,j}$ is parafree of rank $d$, and so its image in $K_+^{\rm ab}$
has Pr\"ufer rank at most $d$. It follows that $K_+^{\rm ab}$
has Pr\"ufer rank at most $d$.
Therefore by Proposition \ref{laboursaving} it is enough to prove that $K_+$ is residually $\Ntf$.

First we claim that $\gamma_{c+1}(V_{0,j})=V_{0,j}\cap \gamma_{c+1}(V_{0,k})$
whenever $c>0$ and $j<k$.  Certainly  $\gamma_{c+1}(V_{0,j})\leq V_{0,j}\cap \gamma_{c+1}(V_{0,k})$.  Since $V_{0,j}$, $V_{0,k}$ are parafree of rank $d$, the groups
$V_{0,j}/\gamma_{c+1}(V_{0,j})$, $V_{0,k}/\gamma_{c+1}(V_{0,k})$ are isomorphic and have the same torsion-free rank, $\kappa$, say. The image in the latter of $V_{0,j}$ has finite index by Lemma \ref{bigsubgroups}, and so both it and its isomorphic image
$V_{0,j}/(V_{0,j}\cap \gamma_{c+1}(V_{0,k}))$ have torsion-free rank $\kappa$.  Thus the torsion-free group $V_{0,j}/\gamma_{c+1}(V_{0,j})$ has the same torsion-free rank as its image
$V_{0,j}/(V_{0,j}\cap \gamma_{c+1}(V_{0,k}))$, and our claim follows.

Hence $\gamma_{c+1}(V_{0,j})=V_{0,j}\cap \gamma_{c+1}(K_+)$
for all $c$ and all $j$.  Therefore $\bigcap_c\gamma_c(K_+)$ intersects each $V_{0,j}$ in the trivial subgroup.  Consequently, $K_+$ is residually $\Ntf$,  as required.  \endproof

I. M. Chiswell

SCHOOL OF MATHEMATICAL SCIENCES

QUEEN MARY, UNIVERSITY OF LONDON

MILE END ROAD

LONDON E1 4NS, ENGLAND.

Email: I.M.Chiswell@qmul.ac.uk
\medskip

A. M. W. Glass

QUEENS' COLLEGE

CAMBRIDGE CB3 9ET, ENGLAND

Email: amwg@dpmms.cam.ac.uk
\medskip

John S. Wilson

MATHEMATICAL INSTITUTE

ANDREW WILES BUILDING

WOODSTOCK ROAD

OXFORD OX2 6GG, ENGLAND

Email: John.Wilson@maths.ox.ac.uk


\begin{thebibliography}{99}

\bibitem{B1} G.\ BAUMSLAG. Some aspects of groups with unique roots. {\em Acta Math. \bf 104}
(1960), 217--303.

\bibitem{B2} G.\ BAUMSLAG.  Groups with the same lower central sequence as a relatively free group. I. The
groups. {\em Trans. Amer. Math. Soc. \bf 129} (1967), 308--321.

\bibitem{CDN} A.\ CLAY, C.\ DESMARAIS and P.\ NAYLOR, Testing bi-orderability of knot groups. Preprint,
available at http://arxiv.org/abs/1410.5774.

\bibitem{CR} A.\ CLAY and D. ROLFSEN. Ordered groups, eigenvalues, knots, surgery and L-spaces.
{\em Math.\ Proc.\ Cambridge Philos.\ Soc.} {\bf 152} (2012), 115--129.

\bibitem{DC} D.\ E.\  COHEN. {\em Combinatorial group theory\/$:$ a topological approach}.  London
Mathematical Society Student Texts, {\bf 14} (Cambridge University Press, Cambridge, 1989).

\bibitem{PC} P. M. COHN. {\em Algebra.} Volume 1 (John Wiley \& Sons, London--New York--Sydney, 1974).

\bibitem{CF} R.\ H.\ CROWELL and R.\ H.\ FOX. {\em Introduction to knot theory}. Graduate Texts in
Mathematics {\bf 57} (Springer-Verlag, New York-Heidelberg, 1977).

\bibitem{F} L.\ FUCHS.  {\em Partially ordered algebraic systems}. (Pergamon Press, New York, 1963).

\bibitem{G99} A.\ M.\ W.\ GLASS.  {\em Partially ordered groups}.   Series in Algebra {\bf 7} (World
Scientific Pub. Co., Singapore, 1999).

\bibitem{LRR}  P.\ A.\ LINNELL, A.\ H.\ RHEMTULLA and D.\ P.\ O.\ ROLFSEN,  Invariant group orderings and
Galois conjugates. {\em J.\ Algebra} {\bf 319} (2008), 4891--4898.

\bibitem{LS} R.\ C.\ LYNDON and P.\ E.\ SCHUPP.  {\em Combinatorial group theory}.  Classics in
Mathematics (Springer-Verlag, Heidelberg, 2001).

\bibitem{MKS} W.\ MAGNUS, A.\ KARRASS and D.\ SOLITAR, {\em Combinatorial group theory$:$ presentations of
groups in terms of generators and relations}, $2^{nd}$ revised edition. (Dover Publications, New York, 1976).

\bibitem{MR} R.\ BOTTO MURA and A.\ H.\ RHEMTULLA. {\em Orderable groups}.  Lecture Notes in Pure and
Applied Algebra {\bf 27}  (Marcel Dekker, New York, 1977).

\bibitem{N}
G.\ NAYLOR and D.\ ROLFSEN. Generalized torsion in knot groups. Preprint, available at
http://arxiv.org/abs/1409.5730.

\bibitem{PR}
B.\ PERRON and D.\ ROLFSEN.  On orderability of fibred knot groups.  {\em Math.\ Proc.\ Cambridge Philos.\ Soc.}  {\bf 135} (2003), 147--153.

\bibitem{R}
D.\ J.\ S.\ ROBINSON. {\em A course in the theory of groups}.  Graduate Texts in Math.\ {\bf 80}
(Springer-Verlag, Heidelberg, 1982).

\bibitem{Ro}
D.\ ROLFSEN. {\em Knots and links}. Mathematics Lecture Series {\bf 7} (Publish or Perish, Berkeley 1976).

\end{thebibliography}
\end{document}